\documentclass{amsart}
\usepackage{amsmath}
\usepackage{amssymb}
\usepackage{amsfonts}
\usepackage{amsxtra}

\newtheorem{theorem}{Theorem}[section]
\newtheorem{lemma}[theorem]{Lemma}
\newtheorem{corollary}[theorem]{Corollary}
\newtheorem{proposition}[theorem]{Proposition}
\newtheorem{remmark}[theorem]{Remark}

\newtheorem{example}[theorem]{Example}

\makeatletter
\newcommand{\cp}{\mathop{\operator@font cp}}
\newcommand{\range}{\mathop{\operator@font range}}
\newcommand{\rank}{\mathop{\operator@font rank}}
\newcommand{\dom}{\mathop{\operator@font dom}}
\newcommand{\Real}{\mathop{\operator@font Re}}
\newcommand{\cont}{\mathop{\operator@font Cont_w}}
\newcommand{\alg}{\mathop{\operator@font Alg\,\mathcal{N}}}
\newcommand{\tr}{\mathop{\operator@font tr}}
\newcommand{\sgn}{\mathop{\operator@font sgn}}
\newcommand{\sot}{\mathop{\operator@font SOT}}
\newcommand{\aff}{\mathop{\operator@font Aff}}
\newcommand{\kn}{\mathop{\operator@font \mathcal{K}(\mathcal{N})}}
\newcommand{\rad}{\mathop{\operator@font Rad}}
\renewcommand{\span}{\mathop{\operator@font span}}
\renewcommand{\rad}{\mathop{\operator@font Rad(\mathcal{N})}}
\makeatother

\numberwithin{equation}{section}

\begin{document}
\title{Compact multiplication operators on nest algebras}
\author{G. Andreolas}
\author{M. Anoussis}
\address{\hspace{-0.7em}Department of Mathematics, University of the Aegean, 832\,00
Karlovassi, Samos, Greece}
\email{gandreolas@aegean.gr}
\address{\hspace{-0.7em}Department of Mathematics, University of the Aegean, 832\,00
Karlovassi, Samos, Greece}
\email{mano@aegean.gr}


\begin{abstract}
 Let $\mathcal{N}$ be a nest on a Hilbert space $H$ and $\alg$ the corresponding nest algebra. We obtain a
characterization of the compact and weakly compact multiplication operators defined on nest algebras. This characterization
leads to a description of the closed ideal generated by the compact elements of $\alg$. We also show that there
is no
non-zero weakly compact multiplication operator on $\alg/\alg\cap \mathcal{K}(H)$.
\end{abstract}

\maketitle

\section{INTRODUCTION}

Let $\mathcal{A}$ be a Banach algebra. A \textit{multiplication operator} $M_{a,b}:\mathcal{A}\rightarrow
\mathcal{A}$ corresponding to $a,b\in\mathcal{A}$ is given by $M_{a,b}(x)=axb$. Properties of compact
multiplication operators
have been investigated since 1964 when Vala published his work ``On compact sets of compact operators'' \cite{1964}. Let
$\mathcal{X}$ be a normed space and $\mathcal{B}(\mathcal{X})$ the space of all bounded linear maps from $\mathcal{X}$ into
$\mathcal{X}$. Vala proved that a non-zero multiplication operator
$M_{a,b}:\mathcal{B}(\mathcal{X})\rightarrow\mathcal{B}(\mathcal{X})$
is
compact if and only if the operators $a\in\mathcal{B}(\mathcal{X})$ and $b\in\mathcal{B}(\mathcal{X})$ are both
compact.

This concept was further investigated by Ylinen in \cite{1972} who proved a similar result for abstract C*-algebras. An
element $a$ of a Banach algebra $\mathcal{A}$ is called \textit{compact} if the multiplication operator
$M_{a,a}:\mathcal{A}\rightarrow\mathcal{A}$ is
compact. Ylinen shows that there exists an isometric $*$-representation $\pi$ of a C*-algebra $\mathcal{A}$ on a Hilbert
space $H$ such
that the operator $\pi(a)$ is compact if and only if $a$ is a compact element of $\mathcal{A}$. Ylinen showed in \cite{1975}
that
this is equivalent with the weak compactness of the map $\lambda_a:\mathcal{A}\rightarrow\mathcal{A}$, $\lambda_a(x)=ax$ (or
equivalently of
the map $\rho_a:\mathcal{A}\rightarrow\mathcal{A}$, $\rho_a(x)=xa$).

In the sequel, these results were generalized to various directions. Let $H$ be a Hilbert space. Akemann and Wright showed in
\cite{ake} that a
multiplication operator $M_{a,b}:\mathcal{B}(H)\rightarrow\mathcal{B}(H)$ is weakly compact if and only if either $a$ or
$b$ is a compact operator. A map
$\Phi:\mathcal{A}\rightarrow\mathcal{A}$ is called \textit{elementary} if $\Phi=\sum_{i=1}^m M_{a_i,b_i}$ for some
$a_i,b_i\in\mathcal{A}$, $i=1,\ldots,m$. Fong and Sourour showed that an elementary operator
$\Phi:\mathcal{B}(H)\rightarrow\mathcal{B}(H)$ is compact if and only if there exist compact operators
$c_i,d_i\in\mathcal{B}(H)$, $i=1,\ldots,m$ such that $\Phi=\sum_{i=1}^m M_{c_i,d_i}$ \cite{fs}. This result was expanded by
Mathieu on prime C*-algebras \cite{m} and later on general C*-algebras by Timoney \cite{tim}. In \cite{m} Mathieu
characterizes the weakly compact elementary operators on prime C*-algebras as well.

From the description of the compact elementary operators by Fong and Sourour, the following conjecture arose: \textit{If
$\Phi$ is a compact elementary operator on the Calkin algebra on a separable Hilbert space, then $\Phi=0$.} This conjecture
was confirmed in \cite{apf} by Apostol and Fialkow and by Magajna in \cite{mag}. In \cite{m} Mathieu proves that
if $\Phi$ is weakly compact, then $\Phi=0$ as well. 

The weak compactness of multiplication operators has been studied in a Banach space setting by Saskmann - Tylli and 
Johnson - Schechtman in \cite{st} and \cite{js} respectively. In \cite{st} the authors give some sufficient conditions 
for 
weak compactness of $M_{a,b}:\mathcal{B}(E)\to\mathcal{B}(E)$, where $E$ is a Banach space. They also provide necessary 
and sufficient conditions for weak compactness of $M_{a,b}$ in case
of some concrete Banach spaces. In \cite{js} the authors 
give a classification of weakly compact multiplication operators on 
$\mathcal{B}(L_p(0,1))$, $1<p<\infty$, which in particular answers a question raised in \cite{st}.

The present work is a study of the compactness properties of multiplication operators defined on nest algebras. Note
that the compactness of the inner derivations defined on nest algebras, that is a special class of elementary operators,
have been studied by Peligrad in \cite{pel}. He characterized the weakly compact
derivations of a nest algebra and obtained necessary and sufficient conditions so that a nest algebra admits
compact derivations.  If
$\mathcal{N}$
is
a
nest, we denote by $\alg$ the corresponding nest algebra. In the second section of the paper we prove a necessary and
sufficient condition for the compactness of multiplication operators defined from $\alg$ into $\alg$. We close the section,
showing by example that there
exist compact multiplication operators on $\alg$ that can not be written as multiplication operators with compact symbols.
In the third section we determine the closed ideal generated by the compact elements of a nest algebra.
In the fourth section of the paper we characterize the weakly compact multiplication operators defined on nest
algebras.
In the last section we show that there are not non-zero weakly compact multiplication operators on
$\alg/\alg\cap\mathcal{K}(H)$ exactly as in
the case of Calkin algebra (i.e. when $\mathcal{N}=\{0,H\}$) \cite{apf}, \cite{mag}, \cite{m}. 

Let us introduce some notation and definitions that will be used throughout the paper. If $H$ is a Hilbert space, then
$\mathcal{B}(H)$ is the space of all bounded linear
operators and $\mathcal{K}(H)$ the space of all compact
operators from $H$ into $H$. Let $\mathcal{E}$ be a Banach space and $r$ a positive number. Then, by
$\mathcal{E}_r$ we denote the closed ball of centre $0$ and radius $r$. Let $e,f$ be elements of a
Hilbert space $H$. We denote by $e\otimes f$ the rank one operator on $H$ defined by
$(e\otimes f)(h)=\langle h,e\rangle f.$

Nest algebras form a class
of non-selfadjoint operator algebras that generalize the block upper triangular matrices to an infinite dimensional Hilbert
space context. They were introduced by Ringrose in \cite{ring} and since then, they have been studied by many authors. 
The monograph of Davidson \cite{dav}
is recommended as a reference.
A nest $\mathcal{N}$ is a totally
ordered family of closed subspaces of a Hilbert space $H$ containing $\{0\}$ and $H$, which is closed under intersection and
closed span. If $H$ is a Hilbert space and $\mathcal{N}$ a nest on $H$, then the nest algebra $\alg$ is the algebra of all
operators $T$ such that $T(N)\subseteq N$ for all $N\in\mathcal{N}$. If $(N_{\lambda})_{\lambda\in\Lambda}$ is a family of
subspaces of a Hilbert space, we denote by $\vee \{N_{\lambda}:\lambda\in\Lambda\}$ their closed linear span and by
$\wedge\{N_{\lambda}:\lambda\in\Lambda\}$ their intersection. If $\mathcal{N}$ is a nest and $N\in\mathcal{N}$, then
$N_-=\vee\{N^{\prime}\in\mathcal{N}:N^{\prime}<N\}$.
Similarly we define $N_+=\wedge\{N^{\prime}\in\mathcal{N}:N^{\prime}>N\}$. The subspaces $N\cap N_{-}^{\perp}$ are called
the \textit{atoms} of $\mathcal{N}$. For any $N\in\mathcal{N}$, we denote by $P_N$ the
orthonormal
projection corresponding to $N$. We endow $\mathcal{N}$ with the order topology and $\{P_N:N\in\mathcal{N}\}$ with the
strong operator topology and denote these spaces by $(\mathcal{N},<)$ and $(P_{\mathcal{N}},\sot)$ respectively. The
natural map taking $N$ to $P_N$ is an order preserving homeomorphism of the compact Hausdorff space $(\mathcal{N},<)$ onto
$(P_{\mathcal{N}},\sot)$, \cite[Theorem 2.13]{dav}. We shall identify the subspaces of a nest with the
corresponding orthogonal projections. In this paper we do not distinguish between these subspaces and projections. We shall
frequently use the fact that a rank one operator $e\otimes f$ belongs to a nest algebra, $\alg$, if and only if the exist an
element $N$ of $\mathcal{N}$ such that $e\in N_-^{\perp}$ and $f\in N$, \cite[Lemmas 2.8 and 3.7]{dav}.
Note that the nest algebras are WOT-closed
subalgebras of $\mathcal{B}(H)$ \cite[Proposition 2.2]{dav}. Throughout the paper we denote by $\mathcal{N}$ a nest acting
on a Hilbert space $H$ and by $\mathcal{K}(\mathcal{N})$ the ideal of compact operators of $\alg$.

\section{COMPACT MULTIPLICATION OPERATORS}

Let $H$ be a Hilbert space and $a$, $b$ elements of $\mathcal{B}(H)$. Vala
proved in \cite{1964} that if
 $a,b\in\mathcal{B}(H)-\{0\}$, then the map $\phi:\mathcal{B}(H)\rightarrow \mathcal{B}(H)$, $x\mapsto axb$
is compact if and only if the operators $a$ and
$b$ are both compact. However, such a result does not hold for nest algebras. Let $\mathcal{N}$ be a nest containing a
projection $P$ such that $\dim(P)=\dim(P^{\perp})=\infty$ and $a\in\alg$ be a
non-compact operator such that $a=PaP^{\perp}$. Then, the multiplication operator
\begin{eqnarray*}
 M_{a,a}:\alg &\rightarrow& \alg, \\
 x &\mapsto axa
\end{eqnarray*}
coincides with the multiplication operator $M_{0,0}$, since
\begin{equation*}
 M_{a,a}(x)=axa=PaP^{\perp}xPaP^{\perp}=0,
\end{equation*} 
for $P^{\perp}xP=0$.

Let $a,b\in\alg$. We introduce the following projections:
\begin{equation*}
 R_a=\vee \{P\in\mathcal{N}: aP=0\}
\end{equation*} 
and 
\begin{equation*}
 Q_b=\wedge \{P\in\mathcal{N}: P^{\perp} b=0\}.
\end{equation*}

\begin{proposition} \label{zero}
 Let $a,b\in\alg$. Then, $M_{a,b}=0$ if and only if $Q_b\leq R_a$.
\end{proposition}
\begin{proof}
 We observe that if $Q_b\leq R_a$, then for all $x\in\alg$: 
\begin{eqnarray*}
 M_{a,b}(x) &=& axb\\
&=& aR_a^{\perp}xQ_bb\\
&=& aR_a^{\perp}Q_bxQ_bb\\
&=& 0,
\end{eqnarray*}
since $a=aR_a^{\perp}$, $b=Q_bb$ and $R_a^{\perp}Q_b=0$. 

Now, suppose that $R_a<Q_b$. We distinguish two cases:
\begin{enumerate}
 \item There exists a projection $P\in\mathcal{N}$ such that $R_a<P<Q_b$. Then, there exist two norm one vectors $e\in
P_-^{\perp}$ and $f\in P$ such that $a(f)\neq 0$ and $b^*(e)\neq 0$. It follows that, $M_{a,b}(e\otimes
f)=a(e\otimes f)b=b^*(e)\otimes a(f) \neq 0$.
 \item There is not any projection of the nest between $R_a$ and $Q_b$, i.e. $R_{a+}=Q_b$. Then, there exist two norm 
one
vectors $e\in
(R_{a+})^{\perp}_-=R_a^{\perp}$ and $f\in R_{a+}=Q_b$ such that $a(f)\neq 0$ and $b^*(e)\neq 0$. It follows that,
$M_{a,b}(e\otimes
f)=a(e\otimes f)b=b^*(e)\otimes a(f) \neq 0$.\qedhere

\end{enumerate}
\end{proof}

The next theorem gives a necessary and sufficient
condition for a non-zero multiplication operator $M_{a,b}:\alg\rightarrow\alg$,
$M_{a,b}(x)=axb$ to be compact.

\begin{theorem} \label{generalcase}
  Let $a,b\in\alg$ such that $M_{a,b}\neq 0$.
The multiplication operator $M_{a,b}:\alg\rightarrow\alg$ is compact if and only if the
operators
$P_+aP_+$ and $P_-^{\perp}bP_-^{\perp}$ are both compact for all
$P\in\mathcal{N}$, $R_a< P< Q_b$ in the case that $R_{a+}\neq Q_b$ or the operators $Q_b a Q_b$ and
$R_a^{\perp}bR_a^{\perp}$ are both compact in the case that $R_{a+}=Q_b$.
\end{theorem}

\begin{proof}
\hspace{-0.3em}Suppose that $M_{a,b}$ is a non-zero compact multiplication operator.\hspace{-0.15em} From Proposition 
\ref{zero}, it follows that $R_a<
Q_b$. Let
$R_{a+}\neq Q_b$. Then, for all
$P\in\mathcal{N}$ such that
$R_a<P<Q_b$, we see that $aP\neq 0$. Let $(e_n)_{n\in\mathbb{N}}\subseteq P_-^{\perp}$ be a bounded sequence and $f\in
P$ such that $a(f)\neq 0$. The sequence $(M_{a,b}(e_n\otimes f))_{n\in\mathbb{N}}=(b^*(e_n)\otimes
a(f))_{n\in\mathbb{N}}$ has a convergent subsequence and therefore the sequence $(b^*(e_n))_{n\in\mathbb{N}}$ has a
convergent subsequence as well. Thus, the operator $b^*
P_-^{\perp}$ is compact and equivalently the operator $P_-^{\perp}bP_-^{\perp}$ is compact. Notice that 
$(P_+)_-^{\perp}b\neq 0$ since $(P_+)_-\leq P<Q_b$. Let $(f_n)_{n\in\mathbb{N}}\subseteq
P_+$ be a bounded sequence and
$e\in(P_+)_-^{\perp}$ such that $b^*(e)\neq 0$. The sequence $(M_{a,b}(e\otimes
f_n))_{n\in\mathbb{N}}=(b^*(e)\otimes
a(f_n))_{n\in\mathbb{N}}$ has a convergent subsequence and therefore the sequence $(a(f_n))_{n\in\mathbb{N}}$ has a
convergent subsequence as well. Thus, the operator $aP_+=P_+aP_+$ is compact. Now, consider the case in which 
$R_{a+}=Q_b$.
We see that
$aQ_b\neq0$. Let $(e_n)_{n\in\mathbb{N}}\subseteq Q_{b-}^{\perp}=R_a^{\perp}$ be a bounded sequence and $f\in Q_b$ such 
that
$a(f)\neq 0$. The sequence $(M_{a,b}(e_n\otimes f))_{n\in\mathbb{N}}=(b^*(e_n)\otimes a(f))_{n\in\mathbb{N}}$ has a
convergent subsequence and therefore the sequence $(b^*(e_n))_{n\in\mathbb{N}}$ has a
convergent subsequence as well. Thus, the operator $b^*
R_a^{\perp}$ is compact and equivalently the operator $R_a^{\perp}bR_a^{\perp}$ is compact. Notice that
$R_a^{\perp}b\neq 0$. Let $(f_n)_{n\in\mathbb{N}}\subseteq Q_b$ be a bounded sequence and
$e\in Q_{b-}^{\perp}=R_a^{\perp}$ such that $b^*(e)\neq 0$. The sequence $(M_{a,b}(e\otimes
f_n))_{n\in\mathbb{N}}=(b^*(e)\otimes
a(f_n))_{n\in\mathbb{N}}$ has a convergent subsequence and therefore the sequence $(a(f_n))_{n\in\mathbb{N}}$ has a
convergent subsequence as well. Thus, the operator $aQ_b=Q_baQ_b$ is compact.

Now, we prove the opposite direction. First, we suppose that $R_{a+}\neq Q_b$ and for all
$P\in\mathcal{N}$ with $R_a<P<Q_b$, the operators
$P_+aP_+$ and $P_-^{\perp}bP_-^{\perp}$ are compact. The multiplication operator $M_{a,b}$ can be written as follows:
\begin{eqnarray*}
 M_{a,b}(x) &=& axb\\
&=& aP_+xb+aP_+^{\perp}xb\\
&=& aP_+xb+aP_+^{\perp}xP_-^{\perp}b+aP_+^{\perp}xP_-b\\
&=& aP_+xb+aP_+^{\perp}xP_-^{\perp}b+aP_+^{\perp}P_-xP_-b\\
&=& P_+aP_+xb+aP_+^{\perp}xP_-^{\perp}bP_-^{\perp}\\
&=& M_{P_+aP_+,b}(x)+M_{aP_+^{\perp},P_-^{\perp}bP_-^{\perp}}(x),
\end{eqnarray*}
since $aP_+^{\perp}P_-xP_-b=0$. We only show that the multiplication operator
$M_{P_+aP_+,b}$ is compact since the proof of the compactness of $M_{aP_+^{\perp},P_-^{\perp}bP_-^{\perp}}$ is similar.
We distinguish two cases:
\begin{enumerate}
 \item We suppose that $R_{a+}\neq R_a$. Let $S=R_{a+}>R_a$. Then, $R_a<S<Q_b$.
Observing that $S_-=R_a$ it follows that $P_+aP_+=P_+aP_+S_-^{\perp}$. For all $x\in\alg$ it follows that
\begin{eqnarray*}
 M_{P_+aP_+,b}(x) &=& P_+aP_+xb\\
&=& P_+aP_+S_-^{\perp}xb\\
&=& P_+aP_+xS_-^{\perp}bS_-^{\perp}\\
&=& M_{P_+aP_+,S_-^{\perp}bS_-^{\perp}}(x).
\end{eqnarray*}
Thus, the multiplication operator $M_{P_+aP_+,b}=M_{P_+aP_+,S_-^{\perp}bS_-^{\perp}}$ is compact since the operators
$P_+aP_+$
and
$S_-^{\perp}bS_-^{\perp}$ are both compact \cite[Theorem 3]{1964}.
\item
Now, we suppose that $R_{a+}=R_a$. Then, there exists a net $(S_i)_{i\in I}\subseteq\mathcal{N}$ which is SOT-convergent 
to
the
projection $R_a$ and for all $i\in I$ the inequality $R_a<S_i$ is satisfied \cite[Theorem 2.13]{dav}. The compactness of 
the
operator $P_+aP_+$ implies that the net
$(P_+aP_+S_i)_{i\in I}$ converges to zero \cite[Proposition 1.18]{dav}. It follows that for some $\varepsilon>0$, we can
choose a projection
$S\in\mathcal{N}$, with $R_a<S<Q_b$ so that $\|P_+aP_+S_-\|<\varepsilon/\|b\|$. We write the multiplication operator
$M_{P_+aP_+,b}$ as follows:
\begin{equation*}
 M_{P_+aP_+,b} = M_{P_+aP_+S_-^{\perp},b}+M_{P_+aP_+S_-,b}.
\end{equation*}
Given that $\|M_{P_+aP_+S_-,b}\|<\varepsilon$, it suffices to show that the multiplication operator
$M_{P_+aP_+S_-^{\perp},b}$
is compact. For all $x\in\alg$ we deduce that:
\begin{eqnarray*}
M_{P_+aP_+S_-^{\perp},b}(x) &=& P_+aP_+S_-^{\perp}xb\\
&=& P_+aP_+S_-^{\perp}xS_-^{\perp}bS_-^{\perp}\\
&=& M_{P_+aP_+S_-^{\perp},S_-^{\perp}bS_-^{\perp}}(x).
\end{eqnarray*}
Therefore, the multiplication operator $M_{P_+aP_+S_-^{\perp},S_-^{\perp}bS_-^{\perp}}$ is compact since the operators
$P_+aP_+$ and
$S_-^{\perp}bS_-^{\perp}$ are both compact.
\end{enumerate}

Finally, we consider the case where
$R_{a+}=Q_b$ and the operators $Q_b a Q_b$ and $R_a^{\perp}bR_a^{\perp}$ are both compact. Seeing that
$a=aR_a^{\perp}$ and $b=Q_bb$,
the multiplication operator $M_{a,b}$ can be written in the following form:
\begin{eqnarray*}
 M_{a,b}(x) &=& axb\\
&=& aR_a^{\perp}xQ_bb\\
&=& Q_baQ_bR_a^{\perp}xQ_bR_a^{\perp}bR_a^{\perp}\\
&=& M_{Q_baQ_bR_a^{\perp},Q_bR_a^{\perp}bR_a^{\perp}}(x).
\end{eqnarray*}
and therefore $M_{a,b}=M_{Q_baQ_bR_a^{\perp},Q_bR_a^{\perp}bR_a^{\perp}}$ is a compact multiplication operator as the
operators
$Q_baQ_b$ and $R_a^{\perp}bR_a^{\perp}$ are both compact.
\end{proof}

\begin{remmark}
Consider the nest $\mathcal{N}=\{\{0\},H\}$ and let
$a,b\in\alg=\mathcal{B}(H)$ with $a,b\neq0$. From
Theorem \ref{generalcase} it follows that the multiplication operator
$M_{a,b}:\mathcal{B}(H)\rightarrow \mathcal{B}(H)$ is compact if and only if the operators $a$
and
$b$ are both compact. In that case the result coincides with Vala's Theorem.
\end{remmark}

\begin{corollary}
Let $a,b\in \alg$ such that $M_{a,b}\neq 0$. Then, the multiplication operator $M_{a,b}:\alg\rightarrow \alg$
is compact if and only if the multiplication operator $M_{a,b}|_{\kn}:\kn\rightarrow\kn$ is compact.
\end{corollary}
\begin{proof}
 The forward direction is immediate. For the opposite direction we observe that the proof is the same as the proof of 
the
forward direction of Theorem \ref{generalcase}. Therefore, we
deduce that the compactness of $M_{a,b}|_{\kn}$ is equivalent with the assertions of Theorem
\ref{generalcase}.
\end{proof}

\begin{corollary}\label{compmult}
 Let $(P_n)_{n\in\mathbb{N}}$ be a sequence of finite rank projections that
increase to the identity and $\mathcal{N}$ the nest $\{P_n\}_{n=1}^{\infty}\cup\{\{0\},H\}$. Let $a,b\in
\alg$ such that $M_{a,b}:\alg\rightarrow \alg$ is a non-zero multiplication
operator.
Then, $b$ is a compact operator if and only if $M_{a,b}$ is a compact multiplication operator. The set of compact 
elements
of $\alg$ is the ideal $\kn$.
\end{corollary}

Let $\mathcal{A}$ be a C*-algebra and $\Phi$ an elementary operator on $\mathcal{A}$. Timoney proved in \cite[Theorem
3.1]{tim} that $\Phi$
is compact if and only if $\Phi$ can be expressed as $\Phi(x)=\sum_{i=1}^m a_ixb_i$ for $a_i$ and $b_i$ compact elements 
of
$\mathcal{A}$ ($1\leq i\leq m$).
The question that arises is whether a compact multiplication operator defined on a nest algebra can always be written as 
an
elementary operator with
compact symbols i.e., if $M_{a,b}:\alg\rightarrow\alg$ is a compact multiplication operator, then there exist an
$l\in\mathbb{N}$ and compact
operators $c_i,d_i\in
\mathcal{B}(H)$, $i\in\{1,\ldots,l\}$, (where $H$ is the underlying Hilbert space of the nest) such that
$M_{a,b}=\sum_{i=1}^l M_{c_i,d_i}$. Another question is whether a compact multiplication operator
$M_{a,b}:\alg\rightarrow\alg$ can be written as an elementary operator $\sum_{i=1}^l M_{c_i,d_i}$ such that the 
operators
$c_i,d_i\in\alg$  $i\in\{1,\ldots,l\}$ are compact elements of the nest algebra. The following
example shows that both questions have a negative answer.

\begin{example}
Let $H$ be a Hilbert space, $\{e_i\}_{i\in\mathbb{N}}$ an orthonormal sequence of $H$,
$\mathcal{N}=\{[\{e_i:i\in\mathbb{N},i\leq n\}]:n\in\mathbb{N}\}\cup\{\{0\},H\}$ and $b=\sum_{n\in\mathbb{N}}\frac{1}{n}
e_n\otimes e_n$ a compact operator of $\alg$.
Then, the multiplication operator $M_{I,b}$ is compact (Corollary \ref{compmult}).
We suppose that there exist compact operators $c_i,d_i\in\mathcal{B}(H)$, $i=1,\ldots,l$ such that
$M_{I,b}=\sum_{i=1}^{l} M_{c_i,d_i}$ and we shall arrive at a contradiction. We consider the following family of rank 
one
operators,
\begin{equation*}
 \left\{x_{r,s}\right\}_{\substack{r\in\mathbb{N}\\ s\in\mathbb{N}\cup\{0\}\\ s<r}}=\left\{e_r\otimes
e_{r-s}\right\}_{\substack{r\in\mathbb{N}\\ s\in\mathbb{N}\cup\{0\}\\
s<r}}\in\alg.
\end{equation*}
Then,
\begin{equation*}
 M_{I,b}(x_{r,s})=\sum_{i=1}^l M_{c_i,d_i}(x_{r,s})
\end{equation*}
i.e.,
\begin{equation*} 
\sum_{n\in\mathbb{N}}\frac{1}{n} e_n\otimes x_{r,s}(e_n)=\sum_{i=1}^l c_ix_{r,s}d_i
\end{equation*}
or
\begin{equation} \label{cex2}
\frac{1}{r}e_r\otimes e_{r-s}=\sum_{i=1}^l d_i^*(e_r)\otimes c_i(e_{r-s}).
\end{equation}
The relation (\ref{cex2}) implies that 
\begin{equation*}
\langle e_{r-s},\frac{1}{r} e_r\otimes e_{r-s}(e_r)\rangle=\sum_{i=1}^l \langle e_{r-s},d_i^*(e_r)\otimes
c_i(e_{r-s})(e_r)\rangle
\end{equation*}
or
\begin{equation} \label{nordereqn}
 \frac{1}{r}=\sum_{i=1}^l \langle e_r,d_i^*(e_r)\rangle\langle e_{r-s},c_i(e_{r-s})\rangle.
\end{equation}
For all $r\in\mathbb{N}$ and $i\in\{1,\ldots,l\}$ we set $D_{r,i}=\langle e_r,d_i^*(e_r)\rangle$ and $C_{r,i}=\langle
e_{r},c_i(e_{r})\rangle$. We denote the vectors $(D_{r,1},\ldots,D_{r,l})\in\mathbb{C}^l$ and
$(C_{r,1},\ldots,C_{r,l})\in\mathbb{C}^l$ by $D_r$ and $C_r$ respectively for all $r\in\mathbb{N}$. Now, we can write
equation (\ref{nordereqn}) in the form 
\begin{equation} \label{afform}
 \frac{1}{r}=\sum_{i=1}^l D_{r,i}C_{r-s,i}.
\end{equation}
This implies
\begin{equation} \label{newform}
 0=\sum_{i=1}^l D_{r,i}\left(C_{r-s,i}-C_{1,i}\right)
\end{equation}
 The sequence $(\mathcal{V}_n)_{n\in\mathbb{N}}=\left(\span\{C_2-C_1,\ldots,
C_n-C_1\}\right)_{n\in\mathbb{N}}$ of
subspaces
of $\mathbb{C}^l$ is increasing and therefore there exists an $n_0\in\mathbb{N}$ such that 
$\mathcal{V}_{n_0}=\mathcal{V}_n$
for all $n\geq n_0$. Therefore, the following holds for all $n\in\mathbb{N}$.
\begin{equation}\label{affine2}
  0=\sum_{i=1}^l D_{n_0,i}(C_{n,i}-C_{1,i}).
\end{equation}
Since the operators $c_i$, $i=1,\ldots,l$ are compact, the sequence $(C_{n})_{n\in\mathbb{N}}$ converges to $0$.
Taking limits in equation (\ref{affine2}) as $n\rightarrow
\infty$ we obtain $0=-\frac{1}{n_0}$ which is a contradiction.
\end{example}

\section{THE IDEAL GENERATED BY THE COMPACT ELEMENTS}

The set of compact elements of a nest algebra does not form an ideal in general. Let $\mathcal{N}$ be a continuous nest 
and
$P,Q\in\mathcal{N}-\{0,I\}$. Then, from Proposition \ref{zero} we can easily see that there exist non-compact operators 
but
compact elements $a,b$ of $\alg$ such that $a=PaP^{\perp}$, $b=QbQ^{\perp}$, while $M_{a+b,a+b}$ is non-compact.

The next proposition characterizes the nests for which the compact elements form an ideal.

\begin{proposition}
 The set of compact elements of $\alg$ is an ideal if and only if for all
$P,S\in\mathcal{N}-\{0,I\}$, with $P<S$, the dimension of $S-P$ is finite. In that case, the set of compact elements of
$\alg$ is the ideal $\kn+Q\alg Q^{\perp}$, for some $Q\in\mathcal{N}-\{0,I\}$.
\end{proposition}
\begin{proof}
 Suppose that there exist $P,S\in\mathcal{N}-\{0,I\}$, with $P<S$, and $\dim (S-P)=\infty$. Let $a,b$ compact elements 
of
$\alg$ such that $R_a<Q_a\leq P<S\leq R_b<S_b$ and the operator $S_+aS_+$ is not compact. We observe that $R_{a+b}=R_a$ 
and
$Q_{a+b}=Q_b$. The operator $S_+bS_+$ is compact while the operator $S_+aS_+$ is not compact. It follows that the 
operator
$S_+(a+b)S_+$ is not compact and therefore the element $a+b$ is non-compact since $R_{a+b}<S<Q_{a+b}$ (Theorem
\ref{generalcase}). Thus, the set of compact elements of $\alg$ is not an ideal.

Now, suppose that for all $P,S\in\mathcal{N}$, with $P<S$, the dimension of $S-P$ is finite. Let $a$ be a compact 
element of
$\alg$. Then, then exists a projection $R\in\mathcal{N}-\{0,I\}$ such that the operators $RaR$ and $R^{\perp}aR^{\perp}$ 
are
compact from Theorem \ref{generalcase}. Let $Q\in\mathcal{N}-\{0,I\}$, with $Q>R$. Then, the operator
$QaQ=aR+a(Q-R)$ is compact
since $\dim(Q-R)<\infty$ and the operator $aR=RaR$ is compact. Similarly, we observe that the operator
$Q^{\perp}aQ^{\perp}$ is compact since $Q>R$ and the operator $R^{\perp}a=R^{\perp}aR^{\perp}$ is compact. If $Q<R$, it 
is
immediate that the operator $QaQ$ is compact. The operator $Q^{\perp}aQ^{\perp}=R^{\perp}a+(Q^{\perp}-R^{\perp})a$ is 
compact
as well since $\dim(R^{\perp}-Q^{\perp})=\dim(Q-R)<\infty$. It follows that the set of compact elements of $\alg$ is
the ideal $\kn+Q\alg Q^{\perp}$, for some $Q\in\mathcal{N}-\{0,I\}$.
\end{proof}

As we have seen, the set of compact elements of $\alg$ does not form an ideal. However, the norm closed ideal generated 
by
the compact elements of $\alg$ has a nice description.
Let $\mathbf{A}$ be the set of atoms of $\mathcal{N}$. The map $\Delta_{\mathcal{N}}:\alg\rightarrow\alg$,
$x\mapsto\sum_{A_{\alpha}\in\mathbf{A}} A_{\alpha}xA_{\alpha}$ is a projection to the atomic part of the diagonal of
$\alg$. The Jacobson radical of $\alg$ is denoted by $\rad$.

\begin{theorem}
 The ideal $\mathcal{J}_c$, generated by the compact elements of the nest algebra $\alg$, is equal to $\kn+\rad$.
\end{theorem}
\begin{proof}
From \cite[Theorem 11.6]{dav} it follows that the set $\mathcal{K}(\mathcal{N})+\rad$ is a closed
ideal.

 Let $a$ be a compact element of $\alg$. From Proposition \ref{zero} and Theorem \ref{generalcase} it follows that $a$ 
is the
sum of a
compact operator and an operator of the form $PaP^{\perp}$ for some $P\in\mathcal{N}$. Therefore, the operator $a$ 
belongs to
the set $\kn+\rad$.

Now, we prove that $\kn+\rad\subseteq\mathcal{J}_c$. It suffices to show that $\rad$ $\subseteq \mathcal{J}_c$, since 
the
compact operators of $\alg$ are compact elements of the nest algebra. Let $a$ be an element of
$\rad$ and $\varepsilon$ a strictly positive number. Then, there is a finite subnest
$\mathcal{F}=\{0=P_1<P_2<\cdots<P_n=I\}$ of $\mathcal{N}$ such that $\|\Delta_{\mathcal{F}}(a)\|<\varepsilon$ 
\cite[Theorem
6.7]{dav}. We thus have:
\begin{equation*}
  \Delta_{\mathcal{F}}(a)+\sum_{i=2}^n P_iP_{i-1}^{\perp}aP_i^{\perp}= a.
\end{equation*}
It follows that $a$ can be written as a sum of compact elements of $\alg$, $P_iP_{i-1}^{\perp}aP_i^{\perp}$, and
an operator $\Delta_{\mathcal{F}}(a)$ of norm less than $\varepsilon$. Therefore, $a\in\mathcal{J}_c$.
\end{proof}

\begin{corollary}
 If $\mathcal{N}$ is a continuous nest, then $\mathcal{J}_c=\rad$.
\end{corollary}
\begin{proof}
 It is immediate from the fact that a compact operator $c\in\alg$ belongs to the radical if and only if
$\Delta_{\mathcal{N}}(c)=0$
\cite[Corollary 6.9]{dav}.
\end{proof}

\section{WEAKLY COMPACT MULTIPLICATION OPERATORS}

Akemann and Wright give a characterization of certain weakly compact maps on $\mathcal{B}(H)$ in \cite[Proposition
2.1]{ake}. We adjust that result to the case of nest algebras.

\begin{proposition} \label{wcmap}
 Let $\varphi:\alg\rightarrow \alg$ be a bounded linear map which is w*-continuous and maps $\mathcal{K}(\mathcal{N})$ 
into
$\mathcal{K}(\mathcal{N})$. Then,
$\varphi=(\varphi|_{\mathcal{K}(\mathcal{N})})^{**}$ and
$\varphi$ is weakly compact if and only if $\varphi(\alg)\subseteq\mathcal{K}(\mathcal{N})$.
\end{proposition}
\begin{proof}
The steps of the proof are very similar to those of \cite[Proposition 2.1]{ake}. Note that the dual space of
$\mathcal{K}(\mathcal{N})$ is
$\mathcal{L}^1(H)/\mathcal{A}_0$ where $H$ is the the underlying Hilbert space of $\mathcal{N}$, $\mathcal{L}^1(H)$ the
space of the trace class operators on $H$ and $\mathcal{A}_0=\{T\in\mathcal{L}^1(H):P_-^{\perp}TP=0,\ \forall
P\in\mathcal{N}\}$. The second dual of $\mathcal{K}(\mathcal{N})$ is $\alg$ \cite[Theorem 16.6]{dav}. Note that
$(\varphi|_{\mathcal{K}(\mathcal{N})})^{**}:\alg\rightarrow\alg$ is w*-continuous as a dual operator and it
agrees with
the w*-continuous map $\varphi$ on the w*-dense set $\mathcal{K}(\mathcal{N})\subseteq\alg$, \cite[Corollary 3.13]{dav}.
Therefore $\varphi=(\varphi|_{\mathcal{K}(\mathcal{N})})^{**}$ since $\varphi$ and
$(\varphi|_{\mathcal{K}(\mathcal{N})})^{**}$ are w*-continuous and
$\overline{\mathcal{K}(\mathcal{N})}^{w^*}=\alg$.

Now assume that $\varphi$ is weakly compact. Then, $(\varphi|_{\mathcal{K}(\mathcal{N})})^{**}=\varphi$ is weakly 
compact,
whence
$\varphi|_{\mathcal{K}(\mathcal{N})}$ is weakly compact \cite[Theorem 8, p. 485]{ds}. This implies that
$\varphi(\alg)\subseteq\mathcal{K}(\mathcal{N})$,
\cite[Theorem 2, p. 482]{ds}.

Conversely, assume that $\varphi(\alg)\subseteq\mathcal{K}(\mathcal{N})$. The nest algebra $\alg$ is w*-closed
\cite[Proposition
2.2]{dav} and therefore the closed unit ball $(\alg)_1$ is w*-compact. By the w*-continuity of $\varphi$ the set
$\varphi((\alg)_1)$ is w*-compact. Therefore, the set $\varphi((\alg)_1)\subseteq\mathcal{K}(\mathcal{N})$ is weakly 
compact
since the
relative w*-topology of $\mathcal{K}(\mathcal{N})$ coincides with the weak topology on $\mathcal{K}(\mathcal{N})$.
\end{proof}

\begin{corollary} \label{cor1}
 Let $a,b\in\alg$. Then, the multiplication operator $M_{a,b}:\alg\rightarrow \alg$, $x\mapsto axb$ is 
weakly compact if and only if $M_{a,b}(\alg)\subseteq \kn$.
\end{corollary}

\begin{lemma} \label{utility}
 Let $a,b\in\alg$ and $(e_n)_{\mathbb{N}}$,
$(f_n)_{n\in\mathbb{N}}$
orthonormal sequences in $H$ such that $e_n\otimes f_n\in\alg$ for all $n\in\mathbb{N}$. If there exists an 
$\varepsilon>0$
such that $\|a(f_n)\|\geq\varepsilon$ and $\|b^*(e_n)\|\geq\varepsilon$ for all $n\in\mathbb{N}$, then there
exists a strictly increasing sequence $(k_n)_{n\in\mathbb{N}}$ such that the
operator
$a\left(\sum_{n\in\mathbb{N}}e_{k_n}\otimes f_{k_n}\right)b=\sum_{n\in\mathbb{N}}b^*(e_{k_n})\otimes a(f_{k_n})\in\alg$ 
is
not compact and for any subsequence $(k_{n_m})_{m\in\mathbb{N}}$ the operator
$\sum_{n\in\mathbb{N}}b^*(e_{k_{n_m}})\otimes a(f_{k_{n_m}})\in\alg$ is non-compact as well.
\end{lemma}
\begin{proof}
First we construct the sequence $(k_n)_{n\in\mathbb{N}}\subseteq \mathbb{N}$ by induction. We set $k_1=1$. Suppose we 
have
determined $k_i$ for all $i\in\{2,\ldots,n-1\}$ for some $n\in\mathbb{N}$. Then, we choose $k_n\in\mathbb{N}$ such that 
\begin{equation}\label{sub1}
 |\langle a(f_{k_m}),a(f_{k_n})\rangle|=|\langle
a^*a(f_{k_m}),f_{k_n}\rangle|<\frac{\varepsilon^2}{3\cdot 2^n}
\end{equation}
and 
\begin{equation}\label{sub2}
 |\langle b^*(e_{k_m}),b^*(e_{k_n})\rangle|=|\langle bb^*(e_{k_m}),e_{k_n}\rangle|<\frac{\varepsilon^2}{3\cdot 2^n}
\end{equation}
for all $m\in\mathbb{N}$, with $m<n$.

We suppose that the
operator $\sum_{n\in\mathbb{N}}b^*(e_{k_n})\otimes a(f_{k_n})$ is compact. Then, the operator
$a^*\left(\sum_{n\in\mathbb{N}}b^*(e_{k_n})\otimes a(f_{k_n})\right)b^*$ is compact as
well and therefore, there exists a $m_0\in\mathbb{N}$ such that,
\begin{eqnarray}\label{lemeq2}
\frac{\varepsilon^4}{2} &> \left|\left\langle a^*\left(\sum_{n\in\mathbb{N}}b^*(e_{k_n})\otimes
a(f_{k_n})\right)b^*(e_{k_{m_0}}),f_{k_{m_0}}\right\rangle\right|\nonumber \\
&= \left|\left\langle \left(\sum_{n\in\mathbb{N}}b^*(e_{k_n})\otimes
a(f_{k_n})\right)b^*(e_{k_{m_0}}),a(f_{k_{m_0}})\right\rangle\right|\\
 &= \left|\sum_{n\in\mathbb{N}}\left\langle b^*(e_{k_{m_0}}),b^*(e_{k_n})\right\rangle\left\langle
a(f_{k_n}),a(f_{k_{m_0}})\right\rangle\right|\nonumber
\end{eqnarray}
For all $n\in\mathbb{N}$ we set 
\begin{eqnarray*}
\lambda_n &=& \langle b^*(e_{k_{m_0}}), b^*(e_{k_n})\rangle\\
\mu_n &=& \langle a(f_{k_n}), a(f_{k_{m_0}})\rangle.
\end{eqnarray*}
Note that 
$\lambda_{m_0}=\langle b^*(e_{k_{m_0}}), b^*(e_{k_{m_0}})\rangle =\|b^*(e_{k_{m_0}})\|^2\geq\varepsilon^2$ and similarly
$\mu_{m_0}\geq\varepsilon^2$. From the inequalities (\ref{sub1}) and (\ref{sub2}) it follows that
\begin{equation*}
 |\lambda_n||\mu_n|<\left(\frac{\varepsilon^2}{3\cdot 2^{m_0}}\right)^2
\end{equation*}
for all $n<m_0$ and
\begin{equation*}
 |\lambda_n||\mu_n|<\left(\frac{\varepsilon^2}{3\cdot 2^{n}}\right)^2
\end{equation*}
for all $n>m_0$. Thus, the inequality (\ref{lemeq2}) implies that
\begin{eqnarray*}
 \frac{\varepsilon^4}{2}&>&\left|\sum_{n\in\mathbb{N}}\lambda_n \mu_n\right|\\
&\geq& \lambda_{m_0} \mu_{m_0}-\left|\sum_{n\neq m_{0}}\lambda_n \mu_n\right|\\
&\geq& \varepsilon^4-\sum_{n\neq m_{0}}\left|\lambda_n \mu_n\right|\\
&=& \varepsilon^4 - \sum_{n< m_{0}}\left|\lambda_n \mu_n\right|-\sum_{n> m_{0}}\left|\lambda_n \mu_n\right|\\
&>& \varepsilon^4-(m_0-1)\left(\frac{\varepsilon^2}{3\cdot 
2^{m_0}}\right)^2-\sum_{n>m_0}\left(\frac{\varepsilon^2}{3\cdot
2^n}\right)^2\\
&>& \varepsilon^4-\frac{\varepsilon^4}{9}\sum_{n\in\mathbb{N}}\frac{1}{2^n}\\
&>v \varepsilon^4-\frac{\varepsilon^4}{9}=\frac{8\varepsilon^4}{9},
\end{eqnarray*}
which is a contradiction and therefore the operator $a\left(\sum_{n\in\mathbb{N}}e_{k_n}\otimes f_{k_n}\right)b\in\alg$ 
is
not compact. It is obvious that we can follow the above steps of the proof for all subsequences
$(k_{n_m})_{m\in\mathbb{N}}$ of $(k_n)_{n\in\mathbb{N}}$. Therefore, the operator
$\sum_{n\in\mathbb{N}}b^*(e_{k_{n_m}})\otimes a(f_{k_{n_m}})\in\alg$ is non-compact as well.
\end{proof}

The following lemma provides us with a sufficient condition for the weak compactness of a multiplication operator.

\begin{lemma} \label{suff}
 Let $a,b\in\alg$. If there exists a projection $P\in\mathcal{N}$ such 
that the operators $PaP$ and $P^{\perp}bP^{\perp}$ are both compact, then the multiplication operator
$M_{a,b}:\alg\rightarrow \alg$,
$x\mapsto axb$ is weakly compact.
\end{lemma}
\begin{proof}
 Suppose that there exists a projection $P\in\mathcal{N}$ such that the 
operators $PaP$ and $P^{\perp}bP^{\perp}$ are both compact. Let $x\in\alg_1$. 
Then,
\begin{eqnarray*}
 M_{a,b}(x) &=& axb\\
&=& 
(PaP+PaP^{\perp}+P^{\perp}aP^{\perp})x(PbP+PbP^{\perp}+P^{\perp}bP^{\perp})\\
&=& PaPxb+(PaP^{\perp}+P^{\perp}aP^{\perp})xP^{\perp}bP^{\perp}\\
&=& M_{PaP,b}(x)+M_{(PaP^{\perp}+P^{\perp}aP^{\perp}),P^{\perp}bP^{\perp}}(x).
\end{eqnarray*}
It follows that the multiplication operators $M_{PaP,b}$ and 
$M_{(PaP^{\perp}+P^{\perp}aP^{\perp}),P^{\perp}bP^{\perp}}$ are weakly compact 
since the operators $PaP$ and $P^{\perp}bP^{\perp}$ are both compact (Corollary \ref{cor1}).
\end{proof}

The next lemma gives a necessary condition for the weak compactness of a multiplication operator.

\begin{lemma} \label{nec}
 Let $a,b\in\alg$. If the multiplication operator
$M_{a,b}:\alg\rightarrow \alg$,
$x\mapsto axb$ is weakly compact, then for all $P\in\mathcal{N}$, either the operator $PaP$ is compact or the operator
$P^{\perp}b P^{\perp}$ is compact.
\end{lemma}
\begin{proof}
 Let $P\in\mathcal{N}$. It follows that the multiplication operator 
 \begin{equation*}
  M_{a,b}:P\alg P^{\perp}\rightarrow P\alg P^{\perp}
 \end{equation*}
is weakly
compact or equivalently the multiplication operator 
\begin{equation*}
 M_{PaP,P^{\perp}bP^{\perp}}:\mathcal{B}(H)\rightarrow \mathcal{B}(H)
\end{equation*}
is weakly
compact. Therefore, either the operator $PaP$ is compact or the operator $P^{\perp}bP^{\perp}$ is compact 
\cite[Proposition
2.3]{ake}.
\end{proof}

Now, we proceed to the main theorem of this section. To do so, we introduce the following projections:
\begin{equation*}
 U_a=\vee\{P\in\mathcal{N}: PaP \textrm{ is a compact operator}\}
\end{equation*}
and
\begin{equation*}
 L_b=\wedge\{P\in\mathcal{N}: P^{\perp}bP^{\perp} \textrm{ is a compact operator}\},
\end{equation*}
where $a,b\in\alg$.

\begin{theorem} \label{mainweak}
  Let $a,b\in\alg$. The multiplication operator $M_{a,b}:\alg\rightarrow \alg$, $x\mapsto axb$ is 
weakly compact if and only if one of the following conditions is satisfied:
\begin{enumerate}
 \item $U_a>L_b$.\label{1}
 \item $U_a=L_b=S$ and the operators $SaS$ and $S^{\perp}bS^{\perp}$ are both compact.\label{2}
 \item $U_a=L_b=S$, the operator $SaS$ is compact, the operator $S^{\perp}bS^{\perp}$ is non-compact and for any
$\varepsilon>0$, there exists a projection $P\in\mathcal{N}$,
$P>S$ such that $\|a(P-S)\|<\varepsilon$.\label{3}
 \item $U_a=L_b=S$, the operator $S^{\perp}bS^{\perp}$ is compact, the operator $SaS$ is non-compact and for any
$\varepsilon>0$, there exists a projection
$P\in\mathcal{N}$, $P<S$ such that $\|(S-P)b\|<\varepsilon$.\label{4}
\end{enumerate}
\end{theorem}
\begin{proof}
Suppose that $U_a>L_b$. If there exist a projection $P\in\mathcal{N}$ such that $U_a>P>L_b$, then the
operators $PaP$ and $P^{\perp}bP^{\perp}$ are both compact. If $U_a L_b^{\perp}$ is an atom, then the operators 
$U_aaU_a$ and
$L_b^{\perp}bL_b^{\perp}$ are both compact and therefore the operator $U_a^{\perp}bU_a^{\perp}$ is compact as well since
$U_a=L_{b+}$. Thus, the multiplication operator $M_{a,b}$ is weakly compact (Lemma \ref{suff}).

If condition (\ref{2}) holds, the weak compactness of $M_{a,b}$ follows from Lemma \ref{suff} as well. 

 We suppose that condition (\ref{3}) is satisfied. Let $\varepsilon>0$ and $P\in\mathcal{N}$,
$P>S$ such that $\|a(P-S)\|<\varepsilon$. Then
\begin{equation*}
 a=aS+a(P-S)+aP^{\perp}
\end{equation*}
and
\begin{equation*}
 b=Sb+(S^{\perp}-P^{\perp})b+P^{\perp}b.
\end{equation*}
Let $x\in (\alg)_1$. Then
\begin{eqnarray*}
 M_{a,b}(x) &=& axb\\
&=& (aS+a(P-S)+aP^{\perp})xb\\
&=& aSxb+a(P-S)xb+aP^{\perp}x(Sb+(S^{\perp}-P^{\perp})b+P^{\perp}b)\\
&=& aSxb+a(P-S)xb+aP^{\perp}xP^{\perp}b.
\end{eqnarray*}
The operator $M_{a,b}(x)$ is compact since the operators $aS$ and $P^{\perp}b$ are both compact and
$\|a(P-S)xb\|<\varepsilon$. The multiplication operator $M_{a,b}$ is weakly
compact as the set of weakly compact operators is norm closed \cite[II.C \S6]{woj}.

\hspace{0.05em}Condition (\ref{4}) is symmetric to condition (\ref{3}) and the proof is similar.

\hspace{0.05em}Now, suppose that the multiplication operator $M_{a,b}$ is weakly compact. Then, if $U_a<L_b$ we 
distinguish two cases.
\begin{enumerate}
 \item Suppose that there exists a projection $P\in\mathcal{N}$ such that $U_a<P<L_b$. In that case the
operators $PaP$ and $P^{\perp}bP^{\perp}$ are both non-compact which is a contradiction by Lemma \ref{nec}.
 \item Suppose that $U_{a+}=L_b$. Then, the operators $aL_b$ and $L_{b-}^{\perp} b$ are both
non-compact. Let $\varepsilon>0$ and $(e_n)_{n\in\mathbb{N}}\subseteq L_b$, $(f_n)_{n\in\mathbb{N}}\subseteq 
L_{b-}^{\perp}$
be orthonormal sequences such that $\|a(e_n)\|\geq\varepsilon$ and $\|b^*(f_n)\|\geq\varepsilon$ for all 
$n\in\mathbb{N}$,
\cite[Proposition 5.2.1]{de}. Then, there are subsequences $(e_{k_n})_{n\in\mathbb{N}}$ and $(f_{k_n})_{n\in\mathbb{N}}$
such that
the operator $a\left(\sum_{n\in\mathbb{N}}e_{k_n}\otimes f_{k_n}\right)b\in M_{a,b}((\alg)_1)$ is not compact (Lemma
\ref{utility}).
From Corollary \ref{cor1} we conclude that the multiplication operator $M_{a,b}$ is not weakly compact, that is a
contradiction.
\end{enumerate}

Now, we examine the only two possible cases, $U_a>L_b$ and $U_a=L_b=S$. The first one is condition (\ref{1}) of 
this 
theorem, so
we study the second case. In that case, either the operator $SaS$ is compact or the operator $S^{\perp}bS^{\perp}$ is
compact (Lemma \ref{nec}). We suppose that the operator $SaS$ is compact. If the operator $S^{\perp}bS^{\perp}$ is 
compact
as well, the condition (\ref{2}) is satisfied. We shall see that if the operator $S^{\perp}bS^{\perp}$ is not 
compact, then
condition (\ref{3}) holds.

Suppose that the operator $SaS$ is compact, the operator $S^{\perp}b S^{\perp}$ is not compact and there exists an
$\varepsilon_1>0$ such that for all $P\in\mathcal{N}$, with $P>S$, the inequality $\|a(P-S)\|\geq \varepsilon_1$ holds. 
We
observe that $S_+=S$ (if $S_+>S$, the operator $S_+^{\perp}bS_+^{\perp}$ would be compact and then
$S=Q_b=S_+$). The operator $P^{\perp}bP^{\perp}$ is compact, for all $P\in\mathcal{N}$, with $P>S$. It follows that
$\|(P-S)b\|\geq\varepsilon_2$ or equivalently $\|b^*(P-S)\|\geq\varepsilon_2$ for some $\varepsilon_2>0$, since the 
operator
$S^{\perp}bS^{\perp}$ is not compact. Let $(P_n)_{n\in\mathbb{N}}$ be a decreasing sequence with $P_n>S$ for all
$n\in\mathbb{N}$ such that $\sot-\lim_{n\rightarrow\infty} P_n=S$, \cite[Theorem 2.13]{dav}. We set
$\varepsilon=\max\{\varepsilon_1,\varepsilon_2\}$. Then, for all $n\in\mathbb{N}$,
$\|a(P_n-S)\|\geq\varepsilon$ and $\|b^*(P_n-S)\|\geq\varepsilon$. We choose a norm one vector $e_1\in P_1-S$ such that
$\|b^*(P_1-S)e_1\|\geq\frac{2\varepsilon}{3}$. The SOT-convergence of the sequence $(P_n)_{n\in\mathbb{N}}$ implies that
$\lim_{n\rightarrow\infty}\|b^*(P_n-S)(e_1)\|\leq \|b^*\|\lim_{n\rightarrow\infty}\|(P_n-S)(e_1)\|=0$ and therefore, 
there
exists a $k_2\in\mathbb{N}$, $k_2>1$ such that $\|b^*(P_{k_2}-S)(e_1)\|<\frac{\varepsilon}{3}$. Then,
\begin{eqnarray*}
 \frac{2\varepsilon}{3} &\leq& \|b^*(P_1-S)(e_1)\|\\
&=& \|b^*(P_1-P_{k_2})(e_1)+b^*(P_{k_2}-S)(e_1)\|\\
&\leq& \|b^*(P_1-P_{k_2})(e_1)\|+\|b^*(P_{k_2}-S)(e_1)\|\\
&\leq& \|b^*(P_1-P_{k_2})(e_1)\|+\frac{\varepsilon}{3}.
\end{eqnarray*}
It follows that 
\begin{equation*}
 \|b^*(P_1-P_{k_2})(e_1)\|\geq\frac{\varepsilon}{3}.
\end{equation*}
We set $k_1=1$ and we may suppose that $e_1\in P_{k_1}-P_{k_2}$.

Now, we choose a norm one vector $f_1\in P_{k_2}-S$ such that $\|a(P_{k_2}-S)f_1\|\geq \frac{2\varepsilon}{3}$. 
Repeating
the arguments of the previous paragraph, we find a $k_3\in\mathbb{N}$, $k_3>k_2$, such that 
\begin{equation*}
 \|a(P_{k_2}-P_{k_3})f_1\|\geq\frac{\varepsilon}{3},
\end{equation*}
while considering that $f_1\in P_{k_2}-P_{k_3}$. Using these arguments, one can construct by induction a subsequence
$(P_{k_n})_{n\in\mathbb{N}}$ and two orthonormal sequences $(e_n)_{n\in\mathbb{N}}$ and $(f_n)_{n\in\mathbb{N}}$ with 
the
following properties:
\begin{enumerate}
 \item[(i)] $e_n\in P_{2n-1}-P_{2n}$ and $f_n\in P_{2n}-P_{2n+1}$, for all $n\in\mathbb{N}$.
 \item[(ii)] $\|b^*(e_n)\|=\|b^*(P_{2n-1}-P_{2n})(e_n)\|>\frac{\varepsilon}{3}$ and\\ $\|a(f_n)\| =
\|a(P_{2n}-P_{2n+1})(f_n)\|>\frac{\varepsilon}{3}$, for all $n\in\mathbb{N}$.
\end{enumerate}
 Lemma \ref{utility} shows that there exist subsequences $(e_{k_n})_{n\in\mathbb{N}}$ and $(f_{k_n})_{n\in\mathbb{N}}$ 
such
that the operator $a\left(\sum_{n\in\mathbb{N}} e_{k_n}\otimes f_{k_n}\right)b\in M_{a,b}((\alg)_1)$ is
not compact and Proposition \ref{wcmap} leads us to a contradiction. Therefore, condition (\ref{3}) is satisfied.

The proof in the last case (i.e. $S=U_a=L_b$, $S^{\perp}b S^{\perp}$ is compact and $SaS$ is not compact) is
similar to the previous case, and we omit the details.
\end{proof}

\begin{remmark} \label{conditions}
Observe that if we suppose that the multiplication operator $M_{a,b}$ is not weakly compact, then the conditions of 
Lemma
\ref{utility}
are satisfied. We shall use this fact in the proof of Theorem \ref{Calkin}.
\end{remmark}

The next theorem provides an other characterization of weakly compact multiplication operators.

\begin{theorem}
 Let $a,b\in\alg$. The multiplication operator $M_{a,b}:\alg\rightarrow\alg$ is weakly compact if and only if
for all $\varepsilon>0$ there exist two projections $P_1,P_2\in\mathcal{N}$, with $P_1\leq P_2$, such that the operators
$P_1aP_1$ and $P_2^{\perp}bP_2^{\perp}$ are both compact and $\|a(P_2-P_1)\|<\varepsilon$ or 
$\|(P_2-P_1)b\|<\varepsilon$.
\end{theorem}
\begin{proof}
 Let $M_{a,b}$ be a weakly compact multiplication operator. Suppose that $U_a>L_b$ (condition (\ref{1}) of Theorem
\ref{mainweak}). Then,
either there exist a projection $P_1=P_2\in\mathcal{N}$ such that $U_a>P_1=P_2>L_b$ or the operators $U_aaU_a$ and
$U_a^{\perp}bU_a^{\perp}$ are both compact. In the second case we set $P_1=P_2=U_a$. In any case the inequality
$\|a(P_2-P_1)\|<\varepsilon$ is satisfied for all $\varepsilon>0$, while the operators $P_1aP_1$ and
$P_2^{\perp}bP_2^{\perp}$ are both compact. If $U_a=L_b=S$ and the operators $SaS$ and $S^{\perp}bS^{\perp}$ are both 
compact
(condition (\ref{2}) of Theorem \ref{mainweak}), then for $P_1=P_2=S$ it follows that $\|a(P_2-P_1)\|=0$. If either
condition (\ref{3}) 
of
Theorem \ref{mainweak} holds, then for all $P_2\in\mathcal{N}$ with $P_2>S$ the operator $P_2^{\perp}bP_2^{\perp}$ is
compact. Then, for all $\varepsilon>0$ and $P_1=S$, there exists $P_2>S$ such that $\|a(P_2-P_1)\|<\varepsilon$. If 
condition (\ref{4})
of Theorem \ref{mainweak} is satisfied the proof is similar. 

For the opposite direction let $\varepsilon>0$ and $x\in\alg$. Without loss of generality we suppose that $\|a\|\leq1$ 
and
$\|b\|\leq1$. Then, there exist two projections $P_1,P_2\in\mathcal{N}$, with $P_1<P_2$ that satisfy our hypothesis. It
follows that:
\begin{eqnarray*}
 M_{a,b}(x) &=& axb\\
&=& (aP_1+aP_1^{\perp})x(P_2^{\perp}b+P_2b)\\
&=& aP_1x(P_2^{\perp}b+P_2b)+aP_1^{\perp}xP_2^{\perp}b+aP_1^{\perp}xP_2b.
\end{eqnarray*}
The operators $aP_1\!=\!P_1aP_1$ and $P_2^{\perp}b\!=\!P_2^{\perp}bP_2^{\perp}$ are both compact and
$\|aP_1^{\perp}xP_2b\|\!=\|a(P_2-P_1)x(P_2-P_1)b\|\leq \|a(P_2-P_1)\|\|x\|\|(P_2-P_1)b\|<\varepsilon\|x\|$. Therefore, 
the
operator $M_{a,b}$ is weakly compact since the space of weakly compact operators is closed \cite[Theorem 6, p. 52]{woj}.
\end{proof}

\begin{corollary} \label{wcnest}
Let $\mathcal{N}=\{P_n\}_{n\in\mathbb{N}}\cup\{\{0\},H\}$ be a nest consisting 
of a sequence of finite rank projections that increase to the identity, and let 
$a,b\in\alg$. The multiplication operator $M_{a,b}:\alg\rightarrow \alg$, $x\mapsto axb$ is 
weakly compact if and only if either the operator $a$ is compact or the operator $b$ is compact.
\end{corollary}
\begin{proof}
Suppose that neither $a$ nor $b$ is a compact operator. Then, $U_a=L_b=I$, where $I$ is the identity operator. The 
operator
$I^{\perp}bI^{\perp}=0$ is compact, the operator $IaI=a$ is non-compact and there exists an $\varepsilon>0$ such that 
the
inequality $\|(I-P)b\|\geq \varepsilon$ is satisfied for
all $P\in\mathcal{N}$, $P<I$. The last inequality follows from the non-compactness of the operator
$b$. Thus, the multiplication operator $M_{a,b}$ is not weakly compact (Theorem \ref{mainweak}, case (\ref{4})).

The opposite direction is immediate from \cite[Proposition 2.3]{ake}
\end{proof}

If $\mathcal{S}$ is a nonempty subset of the unit ball of a normed space $\mathcal{A}$,
then
the \textit{contractive perturbations} of $\mathcal{S}$ are defined as
$\cp(\mathcal{S})=\left\{x\in \mathcal{A}\ |\ \|x\pm s\|\leq 1\ \forall s\in
\mathcal{S}\right\}$. We shall write $\cp(a)$ instead of $\cp(\{a\})$ for $a\in\mathcal{A}$. One may define contractive
perturbations of higher order by using the recursive
formula $\cp^{n+1}(\mathcal{S})=\cp\left(\cp^{n}(\mathcal{S})\right)$, $n\in\mathbb N$. The second
contractive perturbations, $\cp^2(a)$, were introduced in \cite{1996} to characterize the compact elements of a 
C*-algebra.
Let $\mathcal{N}$ be a nest as in Corollary \ref{wcnest}. The second author and Katsoulis proved in \cite[Theorem
2.7]{1997} that $a\in(\alg)_1$ is a compact operator if and only if the set of its second contractive perturbations,
$\cp^2_{\alg}(a)$, is compact. The next corollary of Theorem \ref{mainweak} complements that result.

\begin{corollary}
 Let $\mathcal{N}$ be a nest as in Corollary \ref{wcnest} and $a\in\alg$. The following are equivalent:
\begin{enumerate}
 \item The set $\cp^2(a)$ is compact.\label{i}
 \item The set $\cp^2(a)$ is weakly compact.\label{ii}
 \item The operator $a$ is compact.
\end{enumerate}
\end{corollary}
\begin{proof}
 The implication (\ref{i})$\Rightarrow$(\ref{2}) is obvious. 

Now, suppose that the set $\cp^2(a)$ is weakly compact. From
\cite[Proposition 1.2]{1996} we know that $M_{a,a}(\alg_{1/2})\subseteq \cp^2_{\alg}(a)$ and therefore
$M_{a,a}: \alg\rightarrow \alg$ is a weakly compact multiplication operator. Therefore Corollary \ref{wcnest} implies 
that
$a$ is a
compact operator.

Let $a$ be a compact operator. Then, the set $\cp^2(a)$ is compact \cite[Theorem 2.7]{1997}.
\end{proof}

\begin{remmark}
 Let $\mathcal{N}$ be a nest as in Corollary \ref{wcnest} and $a,b\in\alg$. From Corollary \ref{compmult} and Corollary
\ref{wcnest} it follows that the multiplication operator $M_{a,b}:\alg\rightarrow \alg$ is weakly compact while being
non-compact if and
only if the operator $a$ is compact and the operator $b$ is non-compact.
\end{remmark}

\begin{remmark}
Let $a,b\in\alg\!+\mathcal{K}(H)$. The algebra
$\alg\!+\mathcal{K}(H)$ is called
the quasitriangular algebra
of $\mathcal{N}$. The multiplication operator 
\begin{equation*}
 M_{a,b}^{QT}:\alg+\mathcal{K}(H)\rightarrow \alg+\mathcal{K}(H)
\end{equation*}
is compact
(weakly
compact) if and only if the operators $a$ and $b$ are both compact (either $a$ or $b$ is compact). 
\end{remmark}
\begin{proof}
 If the operators $a$ and $b$ are both compact (either $a$ or $b$ is compact), the result follows from \cite[Proposition
2.3]{ake}. If the multiplication operator $M_{a,b}^{QT}$ is compact (weakly compact), the restriction
$M_{a,b}^{QT}|_{\mathcal{K}(H)}$ is compact (weakly compact) and therefore, the second dual
$(M_{a,b}^{QT}|_{\mathcal{K}(H)})^{**}=M_{a,b}$ defined on $\mathcal{B}(H)$ is compact (weakly compact \cite[Theorem 8, 
p.
485]{ds}). Therefore, the
operators $a$ and $b$ are both compact (either $a$ or $b$ is compact) \cite[Proposition 2.3]{ake}. Note that these 
arguments
apply to any operator algebra containing the compact operators.
\end{proof}

\section{MULTIPLICATION OPERATORS ON $\alg/\mathcal{K}(\mathcal{N})$}

In this section, we show that there is not any non-zero weakly compact multiplication operator on $\alg$.

\begin{theorem} \label{Calkin}
 Let $a,b\in\alg$ and $\pi:\alg\rightarrow \alg/\kn$ be the quotient map. The multiplication operator
$M_{\pi(a),\pi(b)}:\alg/\mathcal{K}(\mathcal{N})\rightarrow\alg/\mathcal{K}(\mathcal{N})$ is weakly compact if and only 
if
$M_{\pi(a),\pi(b)}=0$.
\end{theorem}
\begin{proof}
We suppose that $M_{\pi(a),\pi(b)}\neq 0$, or equivalently $M_{a,b}(\alg)\nsubseteq\mathcal{K}(\mathcal{N})$. This is 
also
equivalent to the fact that the multiplication operator $M_{a,b}:\alg\rightarrow\alg$ is not weakly compact (Corollary
\ref{cor1}).
We can see that Remark \ref{conditions} and \cite[Proposition 5.2.1]{de} ensure the existence of some
orthonormal sequences $(e_n)_{n\in\mathbb{N}}$ and $(f_n)_{n\in\mathbb{N}}$ that satisfy the conditions of Lemma
\ref{utility} for the operators $a$ and $b$, i.e. $e_n\otimes f_n\in\alg$, $\|a(f_n)\|\geq \varepsilon$ and
$\|b^*(e_n)\|\geq\varepsilon$, for all $n\in\mathbb{N}$ and some $\varepsilon>0$. The subsequences of
$(e_n)_{n\in\mathbb{N}}$ and
$(f_n)_{n\in\mathbb{N}}$ that Lemma \ref{utility} provides are denoted again by the same symbols. 

Let $(A_n)_{n\in\mathbb{N}}$ be a partition of $\mathbb{N}$ such that $A_n$ be an infinite set for all $n\in\mathbb{N}$.
We show that the following map is an isomorphic embedding:
\begin{eqnarray*}
 u:\ell^{\infty}&\rightarrow& M_{\pi(a),\pi(b)}\left(\alg/\mathcal{K}(\mathcal{N})\right)\\
(x_n)_{n\in\mathbb{N}}&\mapsto& M_{\pi(a),\pi(b)}\left(\pi\left(\sum_{n\in\mathbb{N}}x_n\sum_{i\in A_n}e_i\otimes
f_i\right)\right).
\end{eqnarray*}
First of all we see that $u$ is bounded (we assume that $\|a\|\leq1$ and $\|b\|\leq1$). Indeed, for all
$(x_n)_{n\in\mathbb{N}}\subseteq \ell^{\infty}$,
\begin{eqnarray*}
 \left\|u((x_n)_{n\in\mathbb{N}}\right\|_{\alg/\mathcal{K}(\mathcal{N})} &=&
\inf_{K\in\mathcal{K}(\mathcal{N})}\left\|a\left(\sum_{n\in\mathbb{N}}x_n\sum_{i\in A_n}e_i\otimes
f_i\right)b+K\right\|\\
&\leq& \left\|a\left(\sum_{n\in\mathbb{N}}x_n\sum_{i\in A_n}e_i\otimes
f_i\right)b\right\|\\
&\leq& \|a\|\|b\|\left\|\left(\sum_{n\in\mathbb{N}}x_n\sum_{i\in A_n}e_i\otimes
f_i\right)\right\|\\
&\leq& \|(x_n)_{n\in\mathbb{N}}\|_{\infty}.
\end{eqnarray*}
Then, it suffices to prove that $u$ is bounded below, i.e. there is a positive number $\delta$ such that
$\|u((x_n)_{n\in\mathbb{N}}\|_{\alg/\mathcal{K}(\mathcal{N})}\geq\delta \|(x_n)_{n\in\mathbb{N}}\|_{\infty}$,
($(x_n)_{n\in\mathbb{N}}\in\ell^{\infty}$). Let $(x_n)_{n\in\mathbb{N}}$ be a non-zero element of $\ell^{\infty}$ and
$n_0\in\mathbb{N}$ such that $|x_{n_0}|\geq \frac{3}{4}\|(x_n)_{n\in\mathbb{N}}\|_{\infty}$. Then,
\begin{eqnarray}\label{fl}
 \hspace{2em}\|u((x_n)_{n\in\mathbb{N}}\| &=
\inf_{K\in\mathcal{K}(\mathcal{N})}\left\|a\left(\sum_{n\in\mathbb{N}}x_n\sum_{i\in A_n}e_i\otimes
f_i\right)b+K\right\|\nonumber\\
&\geq \left\|\sum_{n\in\mathbb{N}}x_n\sum_{i\in A_n}b^*(e_i)\otimes
a(f_i)+K_{\varepsilon}\right\|-\frac{\varepsilon^4}{9}\|(x_n)_{n\in\mathbb{N}}\|\nonumber\\
&  (\textrm{for some } K_{\varepsilon}\in\mathcal{K}(\mathcal{N}))\nonumber\\
&\geq \left|\left\langle \sum_{n\in\mathbb{N}}x_n\sum_{i\in A_n} b^*(e_i)\otimes
a(f_i)(b^*(e_{i_0})),a(f_{i_0})\right\rangle\right|\\
&  - \left|\left\langle
K_{\varepsilon}(b^*(e_{i_0})),a(f_{i_0})\right\rangle\right|-\frac{\varepsilon^4}{9}\|(x_n)_{n\in\mathbb{N}}
\|\nonumber\\
&= \left|\sum_{n\in\mathbb{N}}x_n\sum_{i\in A_n}\langle b^*(e_{i_0},b^*(e_i)\rangle\langle
a(f_i),a(f_{i_0})\rangle\right|\nonumber\\
&  -\left|\left\langle
a^*K_{\varepsilon}b^*(e_{i_0}),f_{i_0}\right\rangle\right|-\frac{\varepsilon^4}{9}\|(x_n)_{n\in\mathbb{N}}\|,\nonumber
\end{eqnarray}
where $i_0\in A_{n_0}\subseteq\mathbb{N}$ satisfies $\left|\left\langle
a^*K_{\varepsilon}b^*(e_{i_0}),f_{i_0}\right\rangle\right|<\frac{\varepsilon^4}{9}\|(x_n)_{n\in\mathbb{N}}\|$. Such an 
$i_0$
exists since the
operator $K_{\varepsilon}$ is compact and the set $A_{n_0}\subseteq\mathbb{N}$ is infinite. Before we continue our
calculations we set 
\begin{eqnarray*}
 \lambda_i &=& \langle b^*(e_{i_0}),b^*(e_i)\rangle\\
 \mu_i &=& \langle a(f_i),a(f_{i_0})\rangle,
\end{eqnarray*}
for all $i\in\mathbb{N}$. Now, from the estimation of the formula (\ref{fl}) and the proof of Lemma \ref{utility}
we can
write
\begin{eqnarray*}
 \left\|u((x_n)_{n\in\mathbb{N}})\right\| &\geq& \left|\sum_{n\in\mathbb{N}}x_n\sum_{i\in
A_n}\lambda_i\mu_i\right|-\frac{2\varepsilon^4}{9}\|(x_n)_{n\in\mathbb{N}}\|\\
&\geq& \left|x_{n_0}\sum_{i\in A_{n_0}}\lambda_i\mu_i\right|-\left|\sum_{n\neq
n_0}x_n\sum_{i\in A_n}\lambda_i\mu_i\right|-\frac{2\varepsilon^4}{9}\|(x_n)_{n\in\mathbb{N}}\|\\
&\geq&
\frac{3}{4}\left\|(x_n)_{n\in\mathbb{N}}\right\|\frac{8\varepsilon^4}{9}-\|(x_n)_{n\in\mathbb{N}}\|\frac{\varepsilon^4}{
9}
-\frac{2\varepsilon^4}{9}\|(x_n)_{n\in\mathbb{N}}\|\\
&=& \frac{\varepsilon^4}{3}\|(x_n)_{n\in\mathbb{N}}\|.
\end{eqnarray*}
Thus, the map $u$ is an isomorphism. Then, the closed unit ball of the space $u(\ell^{\infty})$ is not weakly compact and
therefore the multiplication operator $M_{\pi(a),\pi(b)}$ is not weakly compact.
\end{proof}

\begin{remmark}
 Let $a,b\in\alg$. Then, the following are equivalent:
\begin{enumerate}
 \item The multiplication operator $M_{\pi(a),\pi(b)}:\alg/\mathcal{K}(\mathcal{N})\rightarrow\alg/\mathcal{K}(\mathcal{N})$
is compact.
 \item The multiplication operator $M_{\pi(a),\pi(b)}:\alg/\mathcal{K}(\mathcal{N})\rightarrow\alg/\mathcal{K}(\mathcal{N})$
is weakly compact.
 \item $M_{\pi(a),\pi(b)}=0$.
 \item $M_{a,b}(\alg)\subseteq \mathcal{K}(H)$.
 \item The multiplication operator $M_{a,b}$ is weakly compact.

\end{enumerate}

\end{remmark}

\vspace{1em}

\noindent \textbf{Acknowledgements.} The authors would like to thank Prof. V. Felouzis for fruitful
discussions.

\bibliographystyle{amsplain}

\begin{thebibliography}{99}

\bibitem{ake} C. A. Akemann and S. Wright, \textit{Compact actions on C*-algebras}, Glasgow Math. J. \textbf{21} (1980), no.
2, 143-149. MR0582123
\bibitem{1996} M. Anoussis and E. G. Katsoulis, \textit{Compact operators and the
geometric structure of C*-algebras}, Proc. Amer. Math. Soc. \textbf{124} (1996),
2115-2122. MR1322911
\bibitem{1997} M. Anoussis and E. G. Katsoulis, \textit{Compact operators and the
geometric structure of nest algebras}, Indiana Univ. Math. J. \textbf{46} (1997), 319-335. MR1444482
\bibitem{apf} C. Apostol and L. Fialkow, \textit{Structural properties of elementary opertors}, Canad. J. Math. 
\textbf{38}
(1986), 1485-1524. MR0873420
\bibitem{dav} K. Davidson, \textit{Nest algebras,} Longman Scientific \& 
Technical, (1988). MR0972978
\bibitem{de} A. Deitmar and S. Echterhoff, \textit{Principles of harmonic analysis}, Springer, (2009). MR2457798
\bibitem{ds} N. Dunford and J. T. Schwartz, \textit{Linear operators, Part I}, Interscience, (1958). MR0117523
\bibitem{fs} C. K. Fong and A. R. Sourour, \textit{On the operator indentity $\sum A_kXB_k\equiv 0$},
 Can. J. Math. \textbf{31} (1979), 845-857. MR0540912
\bibitem{js} W. B. Johnson and G. Schechtman, \textit{Multiplication operators on $L(L^p)$ and $\ell^p$-strictly 
singular operators},  J. Eur. Math. Soc. \textbf{10} (2008), no. 4, 1105-1119. MR2443930
\bibitem{mag} B. Magajna, \textit{A system of operator equations}, Canad. Math. Bull. \textbf{30} (1987), 200-207. MR0889539
\bibitem{m} M. Mathieu, \textit{Elementary operators on prime C*-algebras II},  Glasgow Math. J. 
\textbf{30} (1988), 275-284. MR0964574
\bibitem{meg} R. E. Megginson, \textit{An introduction to Banach space theory,} 
Springer, (1998). MR1650235
\bibitem{pel} C. Peligrad, \textit{Compact derivations of nest algebras}, Proc. Amer. Math. Soc. \textbf{97} (1986),
668-672. MR0845985
\bibitem{ring} J. R. Ringrose, \textit{On some algebras of operators}, Proc. London Math. Soc. (3) \textbf{15} (1965),
61-83. MR0171174
\bibitem{st} E. Saksman and H.-O. Tylli, \textit{Weak compactness of multiplication operators on spaces of bounded 
linear
operators},  Math. Scand. \textbf{70} (1992), no. 1, 91-111. MR1174205
\bibitem{tim} R. M. Timoney, \textit{Some formulae for norms of elementary 
operators,} J. Operator Theory \textbf{57} (2007), 121-145. MR2301939
\bibitem{1964} K. Vala, \textit{On compact sets of compact operators}, Ann. Acad. Sci.
Fenn. Ser. A I No. \textbf{351} (1964). MR0169078
\bibitem{woj} P. Wojtaszczyk, \textit{Banach spaces for analysts}, Cambridge University Press (1991). MR1144277
\bibitem{1972} K. Ylinen, \textit{A note on the compact elements of C*-algebras}, Proc.
Amer. Math. Soc. \textbf{35} (1972), 305-306. MR0296716
\bibitem{1975} K. Ylinen, \textit{Weakly completely continuous elements of C*-algebras}, Proc.
Amer. Math. Soc. \textbf{52} (1975), 323-326. MR0383095 
\end{thebibliography}

\end{document}